\documentclass{amsart}
\usepackage[utf8]{inputenc}
\usepackage{amsmath}
\usepackage{amssymb}
\usepackage{amsthm}
\usepackage{hyperref}

\newtheorem{thm}{Theorem}[section]

\newtheorem{question}[thm]{Question}
\newtheorem{proposition}[thm]{Proposition}
\newtheorem{corollary}[thm]{Corollary}
\newtheorem{theorem}[thm]{Theorem}
\newtheorem{lemma}[thm]{Lemma}

\theoremstyle{definition}

\newtheorem{example}[thm]{Example}

\newtheorem{remark}[thm]{Remark}

\newcommand{\Z}{\mathbb{Z}}
\newcommand{\wrho}{\widehat{\rho}}
\newcommand{\gens}{\mbox{gens}}

\title[On resurgence via asymptotic resurgence]{On resurgence via asymptotic resurgence}
\author[M.~DiPasquale]{Michael DiPasquale}     
\address{Michael DiPasquale\\     
	Department of Mathematics\\     
	Colorado State University}     
\email{Michael.DiPasquale@colostate.edu}
\urladdr{\url{https://midipasq.github.io/}}   
\author[B.~Drabkin]{Ben Drabkin}
\address{Ben Drabkin\\
Department of Mathematics\\
University of Nebraska-Lincoln}
\email{benjamin.drabkin@huskers.unl.edu}
\urladdr{\url{http://www.math.unl.edu/~bdrabkin2/}}

\begin{document}

\begin{abstract}
The resurgence and asymptotic resurgence of an ideal in a polynomial ring are two statistics which measure the relationship between its regular and symbolic powers.  We address two aspects of resurgence which can be studied via asymptotic resurgence.  First, we show that if an ideal has Noetherian symbolic Rees algebra then its resurgence is rational.  Second, we derive two bounds on asymptotic resurgence given a single known containment between a symbolic and regular power.  From these bounds we recover and extend criteria for the resurgence of an ideal to be strictly less than its big height recently derived by Grifo, Huneke, and Mukundan.  We achieve the reduction to asymptotic resurgence by showing that if the asymptotic resurgence and resurgence are different, then resurgence is a maximum instead of a supremum.
\end{abstract}

\thanks{The second author is partially supported by NSF grant DMS-1601024 and Epscor grant OIA-1557417.}

\maketitle

\section{Introduction}

Suppose $I$ is an ideal in a polynomial ring $R=K[x_1,\ldots,x_\ell]$ over a field $K$.  There are different ways to take powers of $I$.  The ordinary power $I^r$ is the ideal generated by all $r$-fold products of elements of $I$; this retains many of the algebraic properties of $I$.  An alternative is the $s$-th \textit{symbolic power} of $I$, defined as $I^{(s)}=\bigcap_{P\in\mbox{\textup{Ass}(I)}} (I^sR_P\cap R)$, where $\textup{Ass}(I)$ is the set of associated primes of $I$.  The symbolic powers of $I$ retain its geometry; if $I$ is a radical ideal and $K$ is algebraically closed of characteristic zero, the Zariski-Nagata theorem shows that $I^{(s)}$ consists of all polynomials which vanish to order $s$ along the variety defined by $I$.

Symbolic and regular powers are at the center of the \textit{containment problem}.  The containment problem is to determine for which positive integers $r,s$ we have $I^{(s)}\subset I^r$.  A seminal result, which we refer to as \textit{uniform containment}, is that $I^{(hr)}\subset I^r$ if $h$ is the maximum height of an associated prime of $I$; this is proved in successively more generality in the papers of Ein, Lazarsfeld, and Smith~\cite{ELS01}, Hochster and Huneke~\cite{HH02}, and Ma and Schwede~\cite{MS18}.

To quantify the containment problem more precisely for individual ideals, Bocci and Harbourne introduced the notion of \textit{resurgence} in~\cite{BH10}, defined as
\[
\rho(I)=\sup\left\lbrace\frac{s}{r}\,\,\middle|\,\, I^{(s)}\not\subset I^r\right\rbrace.
\]
A coarsening of resurgence, called \textit{asymptotic resurgence}, was introduced by Guardo, Harbourne, and Van Tuyl in~\cite{GHV13}:
\[
\widehat{\rho}(I)=\sup\left\lbrace\frac{s}{r}\,\,\middle|\,\, I^{(st)}\not\subset I^{rt}\mbox{ for all } t\gg 0 \right\rbrace.
\]
The previous definitions and remarks give $\widehat{\rho}(I)\le \rho(I)\le h$, but in general the relationship between resurgence and asymptotic resurgence is somewhat mysterious.  

In this paper we leverage a characterization for asymptotic resurgence in terms of integral closures obtained by the first author with Francisco, Mermin, and Schweig~\cite{DFMS18} to study two problems concerning resurgence.  The crucial observation which allows us to pass from asymptotic resurgence to resurgence in both problems is Proposition~\ref{prop:Resurgence}, in which we prove that the supremum in the definition of resurgence can be replaced by a maximum if the asymptotic resurgence and resurgence are different.

The first problem we study is the rationality of resurgence.  Currently there is no example of an ideal whose resurgence is irrational, although a conjecture of Nagata~\cite{N59}, if resolved positively, would yield the existence of many ideals with irrational resurgence (see Remark~\ref{rem:IrrationalResurgence}).  These ideals have symbolic Rees algebras which are not finitely generated.
Conversely, we show in Theorem~\ref{thm:Main} that if the symbolic Rees algebra of an ideal is finitely generated then its resurgence is necessarily rational.  This is a natural extension of the corresponding result for monomial ideals, proved in~\cite{DFMS18}.  

The second problem we study is that of \textit{expected} resurgence introduced in~\cite{GHM19}; an ideal $I$ has expected resurgence if its resurgence is strictly less than its big height.  The uniform containment result of~\cite{ELS01,HH02,MS18} implies that the resurgence of an ideal is less than \textit{or equal to} its big height, but there is no example of an ideal whose resurgence is equal to its big height.  Ideals with expected resurgence satisfy asymptotic versions of containments between their symbolic and ordinary powers that may not hold for small powers.  In particular, an ideal with expected resurgence satisfies the \textit{stable Harbourne conjecture}, namely $I^{(hr-h+1)}\subset I^r$ for all $r\gg 0$ (see~\cite{G18}).  Harbourne conjectured that $I^{(hr-h+1)}\subset I^r$ holds for $r\ge 1$ ~\cite{BDHKKSS09,HH13}; there are now several fascinating counterexamples to this conjecture (the first recorded in~\cite{DST13}), but none to the stable version.

As observed by Harbourne, Kettinger, and Zimmitti in~\cite{HKZ20}, a consequence of Proposition~\ref{prop:Resurgence} is that $I$ has expected resurgence if and only if the \textit{asymptotic} resurgence of $I$ is strictly less than its big height.  Thus we can study expected resurgence using asymptotic resurgence.  Inspired by the methods of Grifo~\cite{G18} and Grifo, Huneke, and Mukundan~\cite{GHM19}, we apply a refined uniform containment due to Johnson~\cite{J14} to give two upper bounds on asymptotic resurgence based on a single containment.  From these we deduce two criteria for expected resurgence which recover and extend those in~\cite{GHM19}.
%


The structure of this note is as follows.  In Section~\ref{sec:Resurgence} we give some background, recall the main results of~\cite{DFMS18}, and show in Proposition~\ref{prop:Resurgence} that if $\rho(I)<\wrho(I)$ then $\rho(I)$ is a maximum instead of a supremum.  In Section~\ref{sec:main} we prove that the resurgence of an ideal with Noetherian symbolic Rees algebra is rational.  We first show that the asymptotic resurgence of an ideal with Noetherian symbolic Rees algebra is rational, then conclude the general result from Proposition~\ref{prop:Resurgence}.  In Section~\ref{sec:Expected} we give two bounds on asymptotic resurgence and deduce subsequent criteria for expected resurgence; we use one of these criteria to show that squarefree monomial ideals have expected resurgence.  We conclude with several related remarks in Section~\ref{sec:remarks}.


\section{Resurgence and asymptotic resurgence}\label{sec:Resurgence}
In this section we review some tools which we use in later sections, including results from~\cite{DFMS18}.  An element $f\in R$ is \textit{integral} over an ideal $I$ if $x=f$ is a solution of an equation of the form $x^k+r_1x^{k-1}+\cdots+r_k=0,$ where $r_i\in I^i$ for $i=0,\ldots,k$ and $k$ is a positive integer.  The integral closure of $I$, written $\overline{I}$, is the set of all elements $f\in R$ which are integral over $I$.  If $I=\overline{I}$ then we say $I$ is \textit{integrally closed}.

If $\mathbf{K}$ is a field, a \textit{discrete valuation} on $\mathbf{K}$ is a homomorphism $\nu:\mathbf{K}^*\to\Z$ satisfying $\nu(x+y)\ge \min\{\nu(x),\nu(y)\}$.  If $\mathbf{K}$ is the fraction field of the polynomial ring $R=K[x_1,\ldots,x_\ell]$ then a valuation is determined by its values on $R$, so we abuse terminology by referring to valuations on $R$ instead of its field of fractions.  An $R$-valuation is a discrete valuation which is non-negative on $R$ -- that is $\nu(f)\ge 0$ for all $f\in R$.  If $I$ is an ideal of $R$ and $\nu$ is an $R$-valuation we write $\nu(I)$ for the minimum value which $\nu$ takes on $I$.  We say $\nu$ is \textit{supported} on $I$ if $\nu(I)\ge 1$.  Valuations give a very useful membership test for integral closures, called the valuative criterion for integral closure: $f\in\overline{I}$ if and only if $\nu(f)\ge\nu(I)$ for every $R$-valuation $\nu$ (clearly it suffices to take $R$-valuations supported on $I$).  We get from this the ideal membership test: $J\subset \overline{I}$ if and only if $\nu(J)\ge \nu(I)$ for every $R$-valuation $\nu$.

In~\cite{DFMS18} it is shown that the limit
$
\widehat{\nu}(I):=\lim_{s\to \infty} \frac{\nu(I^{(s)})}{s}
$
exists for any $R$-valuation.  These constants generalize the \textit{Waldschmidt constant} $\widehat{\alpha}(I)$, where $\alpha$ is the valuation defined by $\alpha(f)=\deg(f)$ and $\deg(f)$ is the total degree of $f$.

The \textit{Rees valuations} of an ideal $I$ are finitely many valuations $\nu_1,\ldots,\nu_t$ so that $f\in \overline{I^r}$ (for \textit{any} positive integer $r$) if and only if $\nu_i(f)\ge r\nu_i(I)$ for $i=1,\ldots,t$.  For more on Rees valuations see~\cite[Chapter~10]{IntegralClosure}.  We will not use any details of Rees valuations beyond their existence and the fact that there are finitely many of them.  In~\cite{DFMS18} the following characterization of asymptotic resurgence is given.

\begin{thm}{\cite[Theorem~4.10]{DFMS18}}\label{thm:AsymptoticResurgenceByIntegralClosures}
Let $I$ be an ideal and let $\nu_1,\ldots,\nu_r$ be the set of Rees valuations for $I$.  Then
\[
\widehat{\rho}(I)=\max\limits_i\left\lbrace\frac{\nu_i(I)}{\widehat{\nu}_i(I)}\right\rbrace=\sup\limits_\nu\left\lbrace\frac{\nu(I)}{\widehat{\nu}(I)}\right\rbrace,
\]
where the supremum is taken over discrete $R$-valuations which are supported on $I$.
\end{thm}

We prove a variation of~\cite[Proposition~4.19]{DFMS18} which shows that, when $\widehat{\rho}(I)<\rho(I)$, $\rho(I)$ is a maximum of finitely many fractions of the form $\{\frac{s}{r}:I^{(s)}\not\subset I^r\}$.  For an ideal $I\subset R$, write $b=b(I)$ for the minimum of the integers $k$ satisfying $\overline{I^{r+k}}\subset I^r$ for all $r\ge 1$.  The Brian\c{c}on-Skoda theorem~\cite[Theorem~13.3.3]{IntegralClosure} guarantees that $b\le n-1$, where $n$ is the number of variables of $R$.

\begin{proposition}\label{prop:Resurgence}
	Suppose that $I\subset R$ is an ideal and suppose $\widehat{\rho}(I)<\rho(I)$.  Then there exist positive integers $s_0,r_0$ so that $I^{(s_0)}\not\subset I^{r_0}$ and $\frac{s_0}{r_0}>\widehat{\rho}(I)$.  Put
	\[
	N=\frac{b\widehat{\rho}(I)}{s_0/r_0-\widehat{\rho}(I)}.
	\]
	Then
	\[
	\rho(I)=\max\limits_{\substack{2\le r< N,\\[3 pt] r+1\le s< (r+b)\widehat{\rho}(I)}}\left\lbrace \frac{s}{r}\,\,\middle|\,\, I^{(s)}\not\subset I^r\right\rbrace.
	\]
	At worst, we may take $b=(n-1)$ in the definition of $N$ by the Brian\c{c}on-Skoda theorem.  Consequently if $\widehat{\rho}(I)<\rho(I)$ then $\rho(I)$ is rational.
\end{proposition}
\begin{proof}
	If $\widehat{\rho}(I)<\rho(I)$, it is clear that there must exist positive integers $s_0,r_0$ so that $I^{(s_0)}\not\subset I^{r_0}$ and $\frac{s_0}{r_0}>\widehat{\rho}(I)$.  Now set $b=b(I)$ and suppose $I^{(s)}\not\subset I^r$.  Then, since $\overline{I^{r+b}}\subset I^r$, $I^{(s)}\not\subset \overline{I^{r+b}}$.  By~\cite[Lemma~4.12]{DFMS18}, $\frac{s}{r+b}<\wrho(I)$.  Re-arranging, we get
	\begin{equation}\label{eq:rhoineq}
	\frac{s}{r}<\left(1+\frac{b}{r}\right)\widehat{\rho}(I).
	\end{equation}
	It follows that if
	\[
	\frac{s_0}{r_0}\ge\left(1+\frac{b}{r}\right)\widehat{\rho}(I),
	\]
	then $\frac{s}{r}<\frac{s_0}{r_0}$.  Solving for $r$ in the above inequality and using our assumption that $s_0/r_0>\widehat{\rho}(I)$, we see that if
	\[
	r\ge\frac{b\wrho(I)}{s_0/r_0-\wrho(I)}=N
	\]
	then $\frac{s}{r}<\frac{s_0}{r_0}\le \rho(I)$.  Hence to determine $\rho(I)$ it suffices to consider $r< N$.  
	Since $\rho(I)\ge 1$, Inequality~\eqref{eq:rhoineq} implies that we only need to consider values of $s$ satisfying $r\le s< (r+b)\widehat{\rho}(I)$.  Clearly we do not need to consider the cases $r=1$ and $r=s$.  This proves the proposition.
\end{proof}

\begin{remark}
Proposition~\ref{prop:Resurgence} is similar to Proposition 4.1.3 of Denkert's thesis \cite{Denk}. Denkert's result implies that if there exists some $c$ such that $I^{(ct)}=(I^{(c)})^t$ for all $t\in\mathbb{N}$ (i.e. the symbolic Rees algebra of $I$ is Noetherian), and if $\wrho(I)< \rho(I)$, then $\rho(I)$ is the maximum of a finite number of ratios $\frac{m}{r}$ where $I^{(m)}\not\subseteq I^r$ and is thus rational.
\end{remark}

If both $\widehat{\rho}(I)$ and $b(I)$ are known and $I^{(s_0)}\not\subset I^{r_0}$ for some $\frac{s_0}{r_0}>\wrho(I)$, then Proposition~\ref{prop:Resurgence} can be used to effectively determine $\rho(I)$ by checking a finite number of containments using the \textsc{SymbolicPowers} package~\cite{SymbolicPowers} in the computer algebra system \textsc{Macaulay2}~\cite{M2}.  Moreover, if $I$ is a monomial ideal then the integer $b(I)$ can be determined using \textsc{Normaliz}~\cite{Normaliz} (we use the \textsc{Macaulay2} interface to \textsc{Normaliz}).  We illustrate this in the next example.

\begin{example}\label{ex:resurgencefromasymptoticresurgence}
Consider the three squarefree monomial ideals listed below (the first is in the polynomial ring with seven variables $a,\ldots,g$ and the next two are in the polynomial ring with ten variables $a,\ldots,j$).
\[
\begin{array}{rl}
I= & \langle abd,bce,cdf,aef,acg,deg,bfg\rangle\\
J= & \langle abc,ade,bdf,cef,agh,bgi,chi,dgj,ehj,fij\rangle\\
K= & \langle cefg, bdfh, afgh, adei, begi, cdhi, abcj, cdgj, behj, afij\rangle
\end{array}
\]
The ideals $I,J,K$ appear as Examples~3.1, 3.2, and 3.3, respectively, in~\cite{DFMS18}.  Using the interface to the package \textsc{Normaliz} in \textsc{Macaulay2}, we can compute that the integral closure of the Rees algebras $R[It],R[Jt],$ and $R[Kt]$ are generated in degrees $2,3,$ and $3$, respectively.  It follows that $\overline{I^{r+1}}= I^{r-1}\overline{I^{2}},\overline{J^{r+1}}= J^{r-2}\overline{J^{3}}$, and $\overline{K^{r+1}}= K^{r-2}\overline{K^{3}}$ for all $r\ge 2$ (see~\cite[Proposition~5.2.5]{IntegralClosure}).  We can check directly in \textsc{Macaulay2} that $\overline{I^2}\subset I,\overline{J^3}\subset J^2$, and $\overline{K^3}\subset K^2$.  Thus $\overline{I^{r+1}}\subset I^r,\overline{J^{r+1}}\subset J^r,$ and $\overline{K^{r+1}}\subset K^r$ for $r\ge 1$, and $b(I)=b(J)=b(K)=1$.

In~\cite{DFMS18} it is computed that $\wrho(I)=9/7$ and $\wrho(J)=\wrho(K)=6/5$.  We can check directly that $I^{(3)}\not\subset I^2,J^{(4)}\not\subset J^3$, and $K^{(3)}\not\subset K^2$.  Let $N_I,N_J,N_K$ be the integers referred to as $N$ in Proposition~\ref{prop:Resurgence} for each of the ideals $I,J,K$.  We compute
\[
N_I=\frac{b(I)\wrho(I)}{3/2-\wrho(I)}=\frac{9/7}{3/2-9/7}=6
\]
Similarly we compute $N_J=9$ and $N_K=10$.  The only pair $(r,s)$ satisfying $2\le r<N_I=6$, $r+1\le s<(r+1)\wrho(I)=(r+1)9/7$, and $s/r>3/2$ is $(3,5)$.  However we can check directly $I^{(5)}\subset I^3$.  So $\rho(I)=3/2$ by Proposition~\ref{prop:Resurgence}.

Likewise the only pairs $(r,s)$ satisfying $2\le r<N_J=9,r+1\le s<(r+1)\wrho(J)=(r+1)6/5,$ and $s/r>4/3$ are $(2,3)$ and $(5,7)$.  We can check directly that $J^{(3)}\subset J^2$ and $J^{(7)}\subset J^5$, so $\rho(J)=4/3$.

Finally, there is no pair $(r,s)$ satisfying $2\le r<N_K=10,r+1\le s<(r+1)\wrho(K)=(r+1)6/5,$ and $s/r>3/2$.  Thus $\rho(K)=3/2$.
\end{example}

\section{The resurgence of an ideal with Noetherian symbolic Rees algebra is rational}\label{sec:main}
Recall that the \textit{symbolic Rees algebra} of an ideal $I$ is defined as
\[
R_s(I):=\bigoplus\limits_{s\ge 0} I^{(s)}t^s\subset R[t],
\]
where by convention we take $I^{(0)}=R$.  If $R_s(I)$ is Noetherian then all symbolic powers of $I$ can be written in terms of finitely many such powers, so it is natural to expect rationality of statistics measuring the relationship between regular and symbolic powers of $I$.  Nevertheless the proof we give is slightly technical, relying on the characterization of asymptotic resurgence in Theorem~\ref{thm:AsymptoticResurgenceByIntegralClosures} using Rees valuations.  Before giving the proof, we offer a discussion to indicate that some tool at the level of Rees valuations is needed to capture the rationality of resurgence of ideals with Noetherian symbolic Rees algebra.  In other words, we need an additional finiteness property beyond the Noetherian assumption for the symbolic Rees algebra.  Along the way we derive an upper bound on the resurgence of ideals with Noetherian symbolic Rees algebra.

At first glance one may wonder if we can arrive at a conclusion akin to Proposition~\ref{prop:Resurgence} when $R_s(I)$ is Noetherian.  That is, perhaps there exist constants $M,N\in\mathbb{N}$ so that
\begin{equation}\label{eq:finite}
\rho(I)=\max_{0\le s\le M,1\le r\le N}\left\lbrace\frac{s}{r}\,\,\middle|\,\, I^{(s)}\not\subset I^r \right\rbrace \mbox{ for some } M,N\in\mathbb{N}.
\end{equation}
However Equation~\eqref{eq:finite} does not typically hold even if $R_s(I)$ is Noetherian.  For instance, if $I$ is a complete intersection then $I^{(s)}\not\subset I^r$ if and only if $s<r$, hence
\[
\rho(I)=\sup\left\lbrace \frac{s}{r}\,\,\middle|\,\, 0\le s<r\right\rbrace=1,
\]
but clearly Equation~\eqref{eq:finite} does not hold.  More generally, it is shown in~\cite{DFMS18} that if $I$ is a \textit{normal} ideal (that is, all of its powers are integrally closed), then $\rho(I)=\widehat{\rho}(I)$ and $I^{(s)}\not\subset I^r$ if and only if $\frac{s}{r}$ is \textit{strictly less} than $\rho(I)$.  Thus it happens quite often that Equation~\ref{eq:finite} does not hold, even if we replace $I^r$ by its integral closure.

Next we consider another approach, relying only on the finiteness encoded in the Noetherian property of the symbolic Rees algebra.  If $I$ has Noetherian symbolic Rees algebra generated in degree at most $n$, then
\begin{equation}\label{eq:NoetherianSymbolicReesAlgebra}
I^{(s)}=\sum\limits_{y_1+2y_2+\cdots+ny_n=s} I^{y_1}(I^{(2)})^{y_2}\cdots (I^{(n)})^{y_n}.
\end{equation}
We use the containments for $I,I^{(2)},\ldots, I^{(n)}$ to predict some containments of $I^{(s)}$ as follows.  Define $\gamma_i$ to be the maximum integer with respect to the property $I^{(i)}\subset I^{\gamma_i}$ for $i=1,\ldots,n$.  That is $I^{(i)}\subset I^{\gamma_i}$ and $I^{(i)}\not\subset I^{\gamma_i+1}$.  Then it is clear that $I^{(s)}\subset I^{\gamma_s}$, where
\[
\gamma_s=\min\{y_1\gamma_1+\cdots+y_n\gamma_n: y_1+2y_2+\cdots +ny_n=s\},
\]
where $y_1,\ldots,y_n$ are non-negative integers.  Linear programming (details are given for a similar problem in Proposition~\ref{prop:waldgeneral}) yields the following proposition.

\begin{proposition}\label{prop:ResurgenceBound}
Suppose $I$ is an ideal of $R$ and $R_s(I)$ is generated in degrees at most $n$.  Let $\gamma_1,\ldots,\gamma_n$ be as above and suppose $v$ is the value of the following linear program (over the \textit{rationals}):
\begin{center}
	\begin{tabular}{ll}
		minimize & $\gamma_1y_1+\gamma_2y_2+\dots+\gamma_ny_n$ \\
		subject to & $y_i\ge 0$ for $i=1,\ldots,n$ \\
		and & $y_1+2y_2+3y_3+\dots ny_n= 1$. 
	\end{tabular}
\end{center}
Then $\rho(I)\le \frac{1}{v}$.
\end{proposition}


If the generating degree of the symbolic Rees algebra is known, then Proposition~\ref{prop:ResurgenceBound} can give an upper bound on the resurgence.  For instance, suppose $I$ is the cover ideal of a graph.  Then~\cite[Theorem~5.1]{HHT07} shows that $R_s(I)$ is generated in degree $2$.  If $I^{(2)}=I^2$, so $R_s(I)$ is equal to the Rees algebra of $I$, then Proposition~\ref{prop:Resurgence} gives $\rho(I)\le 1$, as expected.  If $I^{(2)}\neq I^2$, then the value of the linear program in Proposition~\ref{prop:ResurgenceBound} is $1/2$, so $\rho(I)\le 2$, recovering the uniform containment bound.


We now give a concrete illustration which shows that the constants $\gamma_s$ can be quite a bit less than the maximal integer $\lambda_s$ so that $I^{(s)}\subset I^{\lambda_s}$.  This illustrates that Proposition~\ref{prop:ResurgenceBound} may yield a strict upper bound on resurgence; in particular the rationality of $\rho(I)$ could not be deduced from the integers $\gamma_s$.

\begin{example}Consider the ideal $I=\langle x,y\rangle\cap\langle x,z\rangle\cap \langle y,z\rangle=\langle xy,xz,yz\rangle$ in the polynomial ring $R=K[x,y,z]$.  The relationship between the regular and symbolic powers of $I$ are well understood.  Setting $\mathfrak{m}=\langle x,y,z\rangle$ we have $I^{r}=I^{(r)}\cap \mathfrak{m}^{2r}$.  From this we can see that $I^{(s)}\subset I^r$ if and only if $s\ge\frac{4}{3}r$.  We can read off from this that $\rho(I)=\wrho(I)=4/3$.

It follows from~\cite[Theorem~5.1]{HHT07} that the symbolic Rees algebra of $I$ is generated in degree $2$, thus
\[
I^{(s)}=\sum_{y_1+2y_2=s} I^{y_1}(I^{(2)})^{y_2}.
\]
Using the notation above, $\gamma_1=\gamma_2=1$ and we have 
$
I^{(s)}\subset I^{\gamma_s},
$
where $\gamma_s=\min\{y_1+y_2:y_1+2y_2=s\}$.  This gives us $\gamma_s=\lfloor\frac{s}{2}\rfloor$ and an upper bound of $2$ on the resurgence of $I$.  As we can see from above, $I^{(s)}\subset I^{\lfloor 3/4 s\rfloor}$, so the predicted value from $\gamma_s$ is off by a linear factor of roughly $1/4 s$, and thus will not be sufficient to compute the resurgence of $I$.  Since $I$ is normal, this example also shows that the \textit{asymptotic} resurgence cannot be predicted in this way using~\cite[Theorem~4.2]{DFMS18}.
\end{example}

In order to obtain the rationality of resurgence, we use Equation~\eqref{eq:NoetherianSymbolicReesAlgebra} together with the finer control over containment granted by the use of valuations.  The intuition is that containments between symbolic powers and (integral closures of) regular powers can be reduced to inequalities between valuations evaluated on $I$ and finitely many symbolic powers of $I$ by Equation~\eqref{eq:NoetherianSymbolicReesAlgebra}.  The existence of Rees valuations ensures that only finitely many of these inequalities are needed, and rationality thus follows.

We proceed to consider how the Noetherian hypothesis on the symbolic Rees algebra interacts with valuations.  The following is a generalization of a result of the second author on Waldschmidt constants~\cite[Theorem~3.6]{DG18}.  Let $\nu:R\to \Z$ be an $R$-valuation.  It follows from properties of valuations that, for any ideals $I$ and $J$, $\nu(I+J)=\min\{\nu(I),\nu(J)\}$ and $\nu(IJ)=\nu(I)+\nu(J)$.
\begin{proposition}\label{prop:waldgeneral}
	Let $I$ be a homogeneous ideal such that $R_s(I)$ is generated in degree at most $n$ and $\nu:R\to \Z$ an $R$-valuation.  For each $i\in\{1,\dots,n\}$, let $\nu^{(i)}:= \nu(I^{(i)})$.  Then
	\[
	\widehat{\nu}(I)=\lim_{s\rightarrow\infty}\frac{\nu^{(1)}y_1+\nu^{(2)}y_2+\dots+\nu^{(n)}y_n}{s}=\min_{m\leq n}\frac{\nu(I^{(m)})}{m}
	\] 
	where $y_1,\dots,y_n$ in the first equality are positive integers minimizing $\nu^{(1)}y_1+\nu^{(2)}y_2+\dots+\nu^{(n)}y_n$ with respect to the constraint $y_1+2y_2+3y_3+\dots+ny_n=s$.
\end{proposition}
\begin{proof}
	Since $R_s(I)=R[I,I^{(2)}t^2,\ldots,I^{(n)}t^n]$, we have that \[I^{(s)}=\sum_{y_1+2y_2+3y_3+\dots+ny_n= m}I^{y_1}(I^{(2)})^{y_2}(I^{(3)})^{y_3}\cdots(I^{(n)})^{y_n}.\]  
	Thus, by the properties mentioned above, 
	\[
	\nu(I^{(s)})=\mbox{min}\left\{\nu\left(I^{y_1}(I^{(2)})^{y_2}(I^{(3)})^{y_3}\dots(I^{(n)})^{y_n}\right)\mid y_1+2y_2+3y_3+\dots ny_n= s \right\}.
	\] 
	Setting $\nu^{(i)}= \nu(I^{(i)})$ gives 
	\[
	\nu(I^{(s)})=\min\left\{\nu^{(1)}y_1+\nu^{(1)}y_2+\dots+\nu^{(n)}y_n \mid y_1+2y_2+3y_3+\dots ny_n= s\right\}.
	\]
	Dividing by $s$ and taking $\lim_{s\to\infty}$ proves the first equality.  For the second equality, consider the \textit{rational} linear program 
	
	\begin{center}
		\begin{tabular}{ll}
			minimize&  $\nu^{(1)}y_1+\nu^{(2)}y_2+\dots+\nu^{(n)}y_n$  \\[3 pt]
			subject to& $y_i\ge 0$ for $i=1,\ldots,n$ \\[3 pt]
			 and & $y_1+2y_2+3y_3+\dots ny_n= s$. 
		\end{tabular}
	\end{center}
	
	\noindent The constraints define an $(n-1)$-simplex with vertices $(s,0,\ldots,0)$, $(0,s/2,0,\ldots,0)$, $\ldots$ , $(0,\ldots,s/n)$; hence the minimum of the linear functional $\nu^{(1)}y_1+\nu^{(2)}y_2+\dots+\nu^{(n)}y_n$ is attained at one of these vertices, say $(0,\ldots,s/k,\ldots,0)$.  It follows that the solution to the above linear program is $\dfrac{\nu^{(k)} s}{k}$ for some $1\le k\le n$.  Whenever $k\mid \nu^{(k)} s$ this solution will also be integral.  As this happens for infinitely many $s$, dividing by $s$ and taking $\lim_{s\to \infty}$ proves the second equality.\qedhere
\end{proof}

\begin{remark}\label{rem:Cvalue}
It is well-known that $R_s(I)$ is Noetherian if and only if there is an integer $c$ so that $I^{(cn)}=(I^{(c)})^n$ for all $n\ge 1$ (see for instance~\cite[Theorem~2.1]{HHT07}).  From this it is straightforward to see that $\widehat{\nu}(I)=\lim_{k\to\infty} \frac{\nu(I^{(kc)})}{kc}=\frac{\nu((I^{(c)})^k)}{kc}=\frac{\nu(I^{(c)})}{c}$.  Proposition~\ref{prop:waldgeneral} offers a slight computational advantage as the value of $c$ guaranteed by the Noetherian property of $R_s(I)$ is not always bounded above by the maximum generating degree of $R_s(I)$ -- see the following example.
\end{remark}

\begin{example}
The following example appears as~\cite[Example~5.5]{HHT07}.  Suppose $R=K[x_1,\ldots,x_7]$ and $I$ is the monomial ideal generated by the monomials below:
\[
\begin{array}{cccccc}
{x}_{1}{x}_{2} & {x}_{1}{x}_{3} & {x}_{2}{x}_{3} & {x}_{1}{x}_{4} & {x}_{2}{x}_{4} & {x}_{1}{x}_{5} \\
{x}_{2}{x}_{5} & {x}_{3}{x}_{5} & {x}_{1}{x}_{6} & {x}_{2}{x}_{6} & {x}_{3}{x}_{6} & {x}_{4}{x}_{6} \\
{x}_{1}{x}_{7} & {x}_{2}{x}_{7} & {x}_{4}{x}_{7} & {x}_{5}{x}_{7}
\end{array}
\]
Then $R_s(I)$ is generated as an $R$-algebra by $52$ monomials of degree at most $7$ (see~\cite[Example~5.5]{HHT07}).  However we can check in Macaulay2~\cite{M2} that $(I^{(i)})^2\neq I^{(2i)}$ for any $i=1,\ldots,7$.  Thus the $c$ in Remark~\ref{rem:Cvalue} is strictly larger than the generating degree of $R_s(I)$.
\end{example}

The next corollary, which follows immediately from Proposition~\ref{prop:waldgeneral} and Theorem~\ref{thm:AsymptoticResurgenceByIntegralClosures}, shows that we can compute asymptotic resurgence of an ideal with Noetherian symbolic Rees algebra exactly using Rees valuations of its symbolic powers up to the generating degree.

\begin{corollary}\label{cor:AssResRational}
Suppose $I\subset R$ is an ideal with Noetherian symbolic Rees algebra generated in degree at most $n$ and Rees valuations $\nu_1,\ldots,\nu_r$.  Then
\[
\widehat{\rho}(I)=\max\limits_{1\le i\le r,1\le j\le n}\left\lbrace \frac{j\nu_i(I)}{\nu_i(I^{(j)})} \right\rbrace.
\]
\end{corollary}

\begin{theorem}\label{thm:Main}
Suppose $I\subset R=K[x_1,\ldots,x_d]$ is an ideal with Noetherian symbolic Rees algebra.  Then the resurgence and asymptotic resurgence of $I$ are rational.
\end{theorem}

\begin{proof}
Let $I$ be an ideal of $R$ so that the symbolic Rees algebra $R_s(I)$ is Noetherian.  Corollary~\ref{cor:AssResRational} shows that $\widehat{\rho}(I)$ is rational.  Hence we consider $\rho(I)$.  If $\rho(I)=\widehat{\rho}(I)$, then we are done.  Otherwise Proposition~\ref{prop:Resurgence} guarantees $\rho(I)$ is rational.
\end{proof}

\section{Expected resurgence}\label{sec:Expected}

An ideal $I\subset R$ with big height $h$ is said to have \textit{expected} resurgence if $\rho(I)<h$.  As observed by Harbourne, Kettinger, and Zimmitti in~\cite{HKZ20}, it follows from Proposition~\ref{prop:Resurgence} that $\wrho(I)<h$ implies $\rho(I)<h$. We also give a short proof of this.

\begin{corollary}\label{cor:expected}
If an ideal $I\subset R$ with big height $h$ has $\wrho(I)<h$, then $\rho(I)<h$.
\end{corollary}
\begin{proof}
It suffices to consider the case $\wrho(I)<\rho(I)$.  By Proposition~\ref{prop:Resurgence}, $\rho(I)=\frac{s_0}{r_0}$ for some integers $r_0,s_0$ so that $I^{(s_0)}\not\subset I^{r_0}$.  If $s_0\ge hr_0$ then $I^{(s_0)}\subset I^{r_0}$ by the well-known uniform containment result for symbolic powers~\cite{ELS01,HH02,MS18}.  Thus $s_0<hr_0$ and $\rho(I)<h$.
\end{proof}

\begin{remark}\label{rem:swanson}
Clearly it suffices to replace $h$ in Corollary~\ref{cor:EquivExpected} with any number $\mathfrak{s}$ so that $I^{(r\mathfrak{s})}\subset I^r$ for all $r\ge 1$.
\end{remark}

\begin{corollary}
Suppose $I$ is the edge ideal of a graph.  Then $\rho(I)<2$.
\end{corollary}
\begin{proof}
It follows from the proof of~\cite[Theorem~3.18]{DFMS18} that $I^{(2r)}\subset I^r$ for every $r\ge 1$ and it is evident from~\cite[Corollary~3.14]{DFMS18} that $\wrho(I)<2$.  Thus Corollary~\ref{cor:EquivExpected} and Remark~\ref{rem:swanson} complete the proof.
\end{proof}

\begin{example}\label{ex:3triangles}
Even if $I$ is an edge ideal of a graph, it may happen that $\wrho(I)<\rho(I)$.  We learned in a personal communication with Alexandra Seceleanu of the following example found by Andrew Conner.  Let $R$ be the polynomial ring in the nine variables $a,\ldots,i$ and $I=(ab,ac,bc,de,df,ef,gh,gi,hi)$ (this is the edge ideal of three disjoint triangles).  Then $\wrho(I)=4/3$ by~\cite[Theorem~3.12]{DFMS18},~\cite[Corollary~7.10]{Waldschmidt16}, and the fact that the three-cycle has asymptotic resurgence  $4/3$.  However, $I^{(6)}\not\subset I^4$, so $\rho(I)\ge 3/2$.  (We can show that $\rho(I)=3/2$ using the technique of Example~\ref{ex:resurgencefromasymptoticresurgence}.)
\end{example}

In~\cite{G18} and~\cite{GHM19}, a containment result of Johnson~\cite{J14} is used in a crucial way to show that a single containment for symbolic powers may give rise to infinitely many containments.  Likewise, we use Johnson's result to show that a single containment leads to an upper bound on asymptotic resurgence.  In the remainder of this section we will assume $I$ is radical, which will make the application of~\cite{J14} more straightforward (see Remark~\ref{rem:analyticspread}).  The main difference between Corollaries~\ref{cor:EquivExpected} and~\ref{cor:strictineq} and the criteria of~\cite{GHM19} is that we replace containment in regular powers with containment in integral closure of regular powers.

\begin{theorem}\label{thm:AsymptoticResurgenceBound}
Suppose $I$ is a radical ideal with big height $h$, and suppose there are positive integers $r,s$ so that $I^{(s+1)}\subset \overline{I^r}$.  Then $\wrho(I)\le \frac{s+h}{r}$.
\end{theorem}
\begin{proof}
We use~\cite[Theorem~4.4]{J14} which states that
\[
I^{(a+qh)}\subset \prod_{i=1}^q I^{(a_i+1)},
\]
where $a$ and $q$ are positive integers, $h$ is the big height of $I$, and $a_1,\ldots,a_q$ are integers so that $\sum a_i=a$.  We apply Johnson's result with $a=qs$ and $a_i=s$ for $i=1,\ldots,q$, yielding
\[
I^{(q(s+h))}=I^{(qs+qh)}\subset (I^{(s+1)})^q\subset (\overline{I^r})^q\subset \overline{I^{qr}} \mbox{ for all } q\ge 1.
\]
Applying part (b) of~\cite[Lemma~4.12]{DFMS18} yields that if $I^{(q(s+h))}\subset \overline{I^{qr}}$ for infinitely many $q> 0$ then $\wrho(I)\le \frac{s+h}{r}$.
\end{proof}

\begin{remark}
The bound in Theorem~\ref{thm:AsymptoticResurgenceBound} does not work for resurgence.  That is, if $I^{(s+1)}\subset I^r$, it does not follow that $\rho(I)\le \frac{r+h}{s}$.  To see this, take any ideal for which $\wrho(I)<\rho(I)$ (we have seen four examples of this so far in Examples~\ref{ex:resurgencefromasymptoticresurgence} and~\ref{ex:3triangles}, but the first such example was found in~\cite{DST13}), and select integers $s,r$ so that $\wrho(I)<\frac{s}{r}<\rho(I)$.  Then $I^{(st)}\subset I^{rt}$ for infinitely many choices of positive integer $t$ (by definition of $\wrho(I)$).  Thus we may select an integer $t$ large enough which satisfies both $I^{(st)}\subset I^{rt}$ and $\frac{st+h}{rt}<\rho(I)$.
\end{remark}

\begin{remark}
Theorem~\ref{thm:AsymptoticResurgenceBound} is a generalization of part (3) of~\cite[Theorem~1.2]{GHV13} to arbitrary containments as well as ideals without Noetherian symbolic Rees algebras.  Note that the statistic referred to as $\rho'_a(I)$ in~\cite{GHV13} has recently been shown to be equal to $\wrho(I)$ in~\cite{HKZ20}.
\end{remark}


\begin{corollary}\label{cor:EquivExpected}
A radical ideal $I$ with big height $h$ has expected resurgence if and only if $I^{(rh-h)}\subset \overline{I^{r}}$ for some positive integer $r$.
\end{corollary}
\begin{proof}
If $I$ has expected resurgence then $I^{(hr-C)}\subset I^r$ for all $r\gg 0$ and any non-negative integer $C$ (see~\cite[Remark~2.7]{G18}).  Thus $I^{(hr-h)}\subset I^r\subset \overline{I^r}$ for all $r\gg 0$.  For the converse, it follows from Theorem~\ref{thm:AsymptoticResurgenceBound} that $\wrho(I)\le \frac{rh-h-1+h}{r}=\frac{rh-1}{r}<h$.  By Corollary~\ref{cor:expected}, $\rho(I)<h$ as well.
\end{proof}

\begin{remark}\label{rem:GHM}
An application of~\cite[Theorem~4.4]{J14} yields that if $I^{(rh-h)}\subset I^{r}$ for some $r\ge 1$ then $I^{(sh)}\subset I^q I^s$ for $s\ge hr-h-1$, where $q=\lfloor s/(hr-1) \rfloor$.  Thus a variant of Corollary~\ref{cor:EquivExpected}, where $\overline{I^r}$ is replaced by $I^r$, follows from~\cite[Corollary~2.8]{GHM19}.
\end{remark}

\begin{remark}
Notice that $\frac{s+h}{r}<h$ if and only if $s<rh-h$.  Since $s$ and $r$ are integers, this means $s\le rh-h-1$.  It follows that the containment $I^{(rh-h)}\subset \overline{I^r}$ is the weakest containment we could require to imply expected resurgence using Theorem~\ref{thm:AsymptoticResurgenceBound}.
\end{remark}


\begin{example}\label{ex:LateContainment}
Let $R=K[x_1,\ldots,x_\ell]$ and let $I=\cap_{1\le i<j\le \ell} (x_i,x_j)$.  It follows from~\cite[Theorem~4.8]{GHM13} that $I^r=\overline{I^r}$ for all $r\ge 1$; in particular $\rho(I)=\wrho(I)$ by~\cite[Corollary~4.14]{DFMS18}.  It is also known that $I^{(2r-2)}\subset I^r$ for $r=\ell$ and the containment fails for $r<\ell$~\cite[Example~2.6]{G18}.  By Corollary~\ref{cor:EquivExpected}, $\rho(I)<2$; applying Theorem~\ref{thm:AsymptoticResurgenceBound} to the containment $I^{(2\ell-2)}\subset I^\ell$  gives $\wrho(I)\le \frac{2\ell-1}{\ell}$.  In fact $\rho(I)=\wrho(I)=\frac{2\ell-2}{\ell}$ by~\cite[Theorem~C]{LM15}.
\end{example}

Using Theorem~\ref{thm:AsymptoticResurgenceBound} and a slightly stricter containment, we can modify the conclusion of Theorem~\ref{thm:AsymptoticResurgenceBound} to a strict inequality.  Before giving the statement we discuss some restrictions we can make on the valuations appearing in Theorem~\ref{thm:AsymptoticResurgenceByIntegralClosures}.  Recall that the \textit{center} of an $R$-valuation $\nu:R\to\Z$ is the prime ideal $\mathfrak{m}_\nu=\{f\in R: \nu(f)\ge 1\}$. The union $\bigcup_r\mbox{Ass}(R/\overline{I^r})$ is a finite set, given precisely by the centers of the Rees valuations of $I$ (see~\cite[Corollary~10.2.4]{IntegralClosure}).  Moreover, $\mbox{Ass}(R/\overline{I^i})\subset \mbox{Ass}(R/\overline{I^{i+1}})$ for every integer $i\ge 1$ (~\cite[Proposition~6.8.8]{IntegralClosure}), so the centers of the Rees valuations coincide with $\mbox{Ass}(R/\overline{I^i})$ for any $i\gg 0$.  It follows that the supremum in Theorem~\ref{thm:AsymptoticResurgenceByIntegralClosures} need only be taken over valuations whose centers are in the set $\mbox{Ass}(R/\overline{I^i})$ for any $i\gg 0$.

Since $\nu\equiv 0$ on $R\setminus \mathfrak{m}_\nu$, $\nu$ extends naturally to an $R_{\mathfrak{m}_{\nu}}$-valuation where $\nu(r/p)=\nu(r)$ for any $p\notin \mathfrak{m}_\nu$.  If $\nu$ is a valuation whose center is a \textit{minimal} prime $P$ of $\overline{I^i}$ for some $i$, then $P$ is also a minimal prime of $I$ and $\nu(I^{(s)})=\nu(I^{(s)}R_P)=\nu(I^sR_P)=\nu(I^s)$ and hence $\widehat{\nu}(I)=\nu(I)$.  Since we always have $\wrho(I)\ge 1$ (see part (1) of~\cite[Theorem~1.2]{GHV13}), we may take the supremum (respectively maximum) in Theorem~\ref{thm:AsymptoticResurgenceByIntegralClosures} over valuations (respectively Rees valuations) whose center is an embedded prime of $\overline{I^i}$ for some integer $i\gg 0$ (if the center of every Rees valuation is a minimal prime of $I$, then $\wrho(I)=1$).  This motivates the hypothesis in the following upper bound on resurgence, which closely mirrors one of the standard hypotheses in~\cite{GHM19}.

\begin{theorem}\label{thm:asymptoticresurgencebound2}
Let $I\subset R$ be a radical ideal with big height $h$.  Let $P_1,\cdots,P_k$ be the set of all embedded primes of $\overline{I^i}$ for some $i\gg 0$.  Put $J=P_1\cap \cdots\cap P_k$.  If $I^{(s+1)}\subset J \overline{I^r}$ for some positive integers $r$ and $s$, then $\wrho(I)<\frac{s+h}{r}$.
\end{theorem}
\begin{proof}
As in the proof of Theorem~\ref{thm:AsymptoticResurgenceBound}, we take an integer $q$ and apply Johnson's result with $a=qs$ and $a_i=s$ for $i=1,\ldots,q$, yielding
\[
I^{(q(s+h))}=I^{(qs+qh)}\subset (I^{(s+1)})^q\subset (J\overline{I^r})^q\subset J^q\overline{I^{rq}} \mbox{ for all } q\ge 1.
\]
Now let $\nu$ be a Rees valuation whose center is among $\{P_1,\ldots,P_k\}$.  Then $\nu(J)\ge 1$.  Apply this valuation to the above containment yields
\[
\nu(I^{(q(s+h))})\ge \nu(J^q\overline{I^{rq}})=q\nu(J)+qr\nu(I).
\]
Dividing by $q(s+h)$ and taking the limit as $q\to\infty$ yields
\[
\widehat{\nu}(I)=\lim_{q\to\infty} \frac{\nu(I^{(q(s+h))})}{q(s+h)}\ge\lim_{q\to\infty} \frac{q\nu(J)+qr\nu(I)}{q(s+h)}\ge \frac{1+r\nu(I)}{s+h}.
\]
It follows that
\[
\frac{\nu(I)}{\widehat{\nu}(I)}\le \frac{(s+h)\nu(I)}{1+r\nu(I)}<\frac{s+h}{r}.
\]
Since this holds in particular for each of the finitely many Rees valuations of $I$, $\wrho(I)<\frac{s+h}{r}$ by Theorem~\ref{thm:AsymptoticResurgenceByIntegralClosures}.
\end{proof}

As a corollary we recover~\cite[Theorem~3.3]{GHM19} with a slightly weaker hypothesis.

\begin{corollary}\label{cor:strictineq}
Suppose $I$ is a radical ideal of $R$ with big height $h$ and let $J$ be the intersection of all embedded primes of $\overline{I^i}$ for some $i\gg 0$.  If $I$ is an ideal with big height $h$ and $I^{(rh-h+1)}\subset J\overline{I^r}$ for some integer $r>0$ then $\rho(I)<h$.
\end{corollary}
\begin{proof}
Applying Theorem~\ref{thm:asymptoticresurgencebound2} gives $\wrho(I)<\frac{rh-h+h}{r}=h$.  The result now follows from Corollary~\ref{cor:EquivExpected}.
\end{proof}

\begin{remark}
Theorem~\ref{thm:asymptoticresurgencebound2} and Corollary~\ref{cor:strictineq} are more user-friendly if $J=\mathfrak{m}$, the homogeneous maximal ideal of $R$.  This happens, if, for instance, $I$ has codimension $\dim R-1$ or $I$ is a local complete intersection.
\end{remark}

\subsection{Resurgence of squarefree monomial ideals}

In~\cite[Example~8.4.5]{BDHKKSS09}, it is shown that a monomial ideal with big height $h$ satisfies $I^{(hr-h+1)}\subset I^r$ for every $r\ge 1$.  We modify the argument of~\cite[Example~8.4.5]{BDHKKSS09} to show the following.

\begin{proposition}\label{prop:sqfree}
Suppose $I\subset K[x_1,\ldots,x_\ell]$ (with $\ell\ge 2$) is a squarefree monomial ideal with big height $h\ge 2$.  Then $I^{(rh-h)}\subset I^r$ for $r\ge \ell$.
\end{proposition}

\begin{remark}
Proposition~\ref{prop:sqfree} is tight in the sense that there are squarefree monomial ideals so that $I^{(rh-h)}\not\subset I^r$ for $r<\ell$.  See Example~\ref{ex:LateContainment}.
\end{remark}

Before proving Proposition~\ref{prop:sqfree}, we record a couple of lemmas.  If $I$ is a monomial ideal, we denote by $\pi(I)$ the least common multiple of the generators of $I$ and by $I^{[r]}$ the ideal generated by $r$th powers of generators of $I$.  Let $\gens(I)$ denote the set of minimal generators of $I$.  Then $\gens(I+J)=\gens(I)\cup \gens(J)$ and $\gens(I\cap J)=\{\mbox{lcm}(m,n):m\in\gens(I),n\in\gens(J)\}$.  Suppose $I,J,$ and $K$ are monomial ideals.  Then the following facts are easily verified.
\begin{align}
\label{eq:intdist}
I\cap (J+K) = & I\cap J+I\cap K\\
\label{eq:pi}
\pi(I\cap J) = & \mbox{lcm}(\pi(I),\pi(J))=\langle \pi(I)\rangle\cap \langle \pi(J)\rangle\\
\label{eq:bracketdist}
(I\cap J)^{[r]} = & I^{[r]}\cap J^{[r]}.
\end{align}

\begin{lemma}\label{lem:pi}
Suppose $I$ is a squarefree monomial ideal with codimension at least $2$ in $K[x_1,\ldots,x_\ell]$.  If $\pi(I)$ has degree $d$ then $\pi(I)^{r-1}\in I^r$ for $r\ge d$.
\end{lemma}
\begin{proof}
Since $I$ is squarefree, so is $\pi(I)$.  For every variable $x_i$ which divides $\pi(I)$, we can pick a monomial $m_i$ which is not divisible by $x_i$ (since $I$ has codimension at least $2$).  Since $m_i$ divides $\pi(I)$ for every $i$, we have $\prod_{i=1}^d m_i\mid \pi(I)^d/\pi(I)=\pi(I)^{d-1}$.  Thus $\pi(I)^{d-1}\in I^d$.  Since $\pi(I)\in I$, $\pi(I)^{r-1}\in I^r$ for every $r\ge d$.
\end{proof}

\begin{lemma}\label{lem:powerandbracket}
Suppose $I$ and $J$ are squarefree monomial ideals.  Then $I^{r}\cap J^{[r]}\subset (I\cap J)^r$.
\end{lemma}
\begin{proof}
We have $\gens(I^r\cap J^{[r]})=\{\mbox{lcm}(g_1\ldots g_r,h^r): g_1,\ldots,g_r\in I\mbox{ and } h\in J\}$.  It suffices to show that $\mbox{lcm}(g_1\ldots g_r,h^r)=\prod_{i=1}^r \mbox{lcm}(g_i,h)$, since $\mbox{lcm}(g_i,h)\in I\cap J$.  Let $a_{ij}$ be the exponent of $x_i$ in $g_j$ and $b_i$ be the exponent of $x_i$ in $h$.  The exponent of $x_i$ in $\mbox{lcm}(g_1\ldots g_r,h^r)$ is $\max\{\sum_j a_{ij},b_ir\}$.  Since each exponent $a_{ij}$ is either $0$ or $1$, and each $b_i$ is either $0$ or $1$ as well, $\max\{\sum_j a_{ij},b_ir\}=\sum_j \max\{a_{ij},b_i\}$.  The latter is the exponent of $x_i$ in $\prod_{i=1}^r \mbox{lcm}(g_i,h)$.  Since this holds for every variable $x_i$, $\mbox{lcm}(g_1\ldots g_r,h^r)=\prod_{i=1}^r \mbox{lcm}(g_i,h)$ and the result follows.
\end{proof}

\begin{example}
Let $I=\langle x^2z,yz,y^2\rangle$ and $J=\langle x,z^2\rangle$.  The ideal $I\cap J=\langle xyz,xy^2,x^2z,yz^2\rangle$ appears in~\cite[Example~4.6]{CEHH17}. The monomial $x^{2k}y^kz^{2k}$ is in $I^{2k}\cap J^{[2k]}$ but not in $(I\cap J)^{2k}$ (or even its integral closure by degree considerations) for any $k\ge 1$.  Likewise the monomial $x^{2k+2}y^kz^{2k+1}$ is in $I^{2k+1}\cap J^{[2k+1]}$ but not in $(I\cap J)^{2k+1}$ (or its integral closure) for any $k\ge 1$.  So the containment in Lemma~\ref{lem:powerandbracket} can fail for every $r\ge 1$ if $I$ and $J$ are not squarefree, even if the right hand side is replaced by the integral closure of $(I\cap J)^r$.
\end{example}

\begin{proof}[Proof of Proposition~\ref{prop:sqfree}]
First suppose that $I$ has codimension at least two.  Let $I=\cap_{i=1}^p Q_i$ be a decomposition of $I$ where each ideal $Q_i$ is a monomial prime ideal, thus generated by a subset of the variables $x_1,\ldots,x_\ell$.  Since $I$ has big height $h$, each $Q_i$ is generated by at most $h$ variables.  We have $I^{(rh-h)}=\cap_{i=1}^p Q_i^{rh-h}$.  Since $Q_i$ is generated by at most $h$ variables, $Q_i^{rh-h}\subset Q_i^{[r]}+\langle \pi(Q_i^{r-1})\rangle$ for every $i=1,\ldots,p$.  Put $[p]=\{1,\ldots,p\}$.  We have 
\[
\begin{array}{rll}
I^{(rh-h)}= & \bigcap\limits_{i=1}^p Q_i^{rh-h} &\\[10 pt]
\subset & \bigcap\limits_{i=1}^p (Q_i^{[r]}+\langle \pi(Q_i^{r-1})\rangle) & \\[10 pt]
=&\sum\limits_{S\subset [p]} \big(\bigcap\limits_{i\in S} Q_i^{[r]} \big)\cap \big( \bigcap\limits_{j\notin S} \langle \pi(Q_j)^{r-1}\rangle \big) & \mbox{by~\eqref{eq:intdist}}\\[10 pt]
=&\sum\limits_{S\subset [p]} \big(\bigcap\limits_{i\in S} Q_i \big)^{[r]}\cap \langle \pi\big(\bigcap\limits_{j\notin S} Q_j\big)^{r-1}\rangle & \mbox{by~\eqref{eq:bracketdist},~\eqref{eq:pi}}\\[10 pt]
\subset & \sum\limits_{S\subset [p]} \big(\bigcap\limits_{i\in S} Q_i \big)^{[r]}\cap \big( \bigcap\limits_{j\notin S} Q_j\big)^{r} & \mbox{by Lemma~\ref{lem:pi}}\\[10 pt]
\subset & I^r & \mbox{by Lemma~\ref{lem:powerandbracket}}.
\end{array}
\]
Finally, if $I$ has codimension one and $\ell\ge 2$, put $m=\gcd(\gens (I))$.  Then $I=\langle m\rangle \cap I'$, where $I'$ has codimension at least two.  So $I^{(rh-h)}=\langle m^{rh-h}\rangle \cap (I')^{(rh-h)}\subset \langle m^r\rangle \cap (I')^{r}=I^r$.
\end{proof}

\begin{corollary}\label{cor:sqfree}
If $I$ is a squarefree monomial ideal in $K[x_1,\ldots,x_\ell]$ with big height $h\ge 2$ then $\rho(I)<h$ and $\wrho(I)\le \frac{h\ell-1}{\ell}=h-\frac{1}{\ell}$.
\end{corollary}
\begin{proof}
This follows from Proposition~\ref{prop:sqfree} by Theorem~\ref{thm:AsymptoticResurgenceBound} and Corollary~\ref{cor:EquivExpected}.
\end{proof}

\begin{example}
Proposition~\ref{prop:sqfree} does not hold for monomial ideals which are not squarefree.  Let $I=\langle x^4y^2,x^3y^3,x^4z,y^3z^2\rangle$.  Then $I$ has big height $2$ (this will be clear shortly) and $m=x^9y^6z^2$ is the only generator of $I^{(4)}$ which is not contained in $I^3$ (easily checked in \textsc{Macaulay2}).  However $m^2\in I^6$ and hence $I^{(4)}\subset \overline{I^3}$.  Moreover we also have $I^{(6)}\subset I^4$.

The argument used in the proof of Proposition~\ref{prop:sqfree} also fails to show that $I^{(4)}\subset \overline{I^3}$.  To see this, write $Q_1=\langle x^4,y^3\rangle,Q_2=\langle x^3,z^2\rangle,$ and $Q_3=\langle y^2,z\rangle$.  Then $I=Q_1\cap Q_2\cap Q_3$ and $\pi(Q_1)=x^4y^3,\pi(Q_2)=x^3z^2,$ and $\pi(Q_3)=y^2z$.  Write $T_r$ for the intersection $\cap_{i=1}^3 (Q_i^{[r]}+\langle \pi(Q_i^{r-1}) )$ which appears in the proof of Proposition~\ref{prop:sqfree}.  Then
$
T_r=\langle x^{4r},y^{3r},x^{4r-4}y^{3r-3}\rangle\cap\langle x^{3r},z^{2r},x^{3r-3}z^{2r-2}\rangle\cap\langle y^{2r},z^r,y^{2r-2}z^{r-1}\rangle.
$
The monomial $x^{4r-4}y^{3r-3}$ is in $T_r$ for every $r$ but is not in $I^r$ until $r=7$.  Degree considerations for the same monomial show that $T_r\not\subset \overline{I^r}$ until $r=7$ also.  Thus a different argument is needed to show $I^{(4)}\subset \overline{I^3}$.
\end{example}

\begin{question}\label{ques:containment}
If $I$ is a monomial ideal with big height $h$ in $R=K[x_1,\ldots,x_\ell]$, is there a fixed integer $\beta$, depending \textit{only} on $R$, so that $I^{(hr-h)}\subset \overline{I^r}$ for $r\ge\beta$?  Can we take $\beta=\ell$?  We ask the same questions if $I$ is any ideal of $R$.
\end{question}

Question~\ref{ques:containment} is sensitive to characteristic.  If $K=\mathbb{F}_q$, the ideal $I$ of all but one of the $\mathbb{F}_q$-points in $\mathbb{P}^{\ell-1}(\mathbb{F}_q)$ has codimension $\ell-1$ in $K[x_1,\ldots,x_\ell]$ and $\wrho(I)=\rho(I)=\frac{(\ell-1)q-(\ell-1)+1}{q}$~\cite[Theorem~3.2]{DHNSST15}.  Hence, by Theorem~\ref{thm:AsymptoticResurgenceBound}, $I^{(hr-h)}\not\subset \overline{I^r}$ for $r<\frac{q}{\ell-2}$.  Thus the integer referred to as $\beta$ in Question~\ref{ques:containment}, if it exists, may depend on properties of $R$ besides its dimension.

\section{Additional Remarks}\label{sec:remarks}

\begin{remark}\label{rem:RationalResurgence}
The converse of Theorem~\ref{thm:Main} is false: suppose that $s$ is a positive integer, $d=k^2$ where $k$ is a positive integer, and $I$ is the homogeneous ideal in $K[x,y,z]$ defining a set of $d=\binom{s+2}{2}$ distinct generic points.  Nagata shows that the symbolic Rees algebra of $I$ is not finitely generated when $d>9$~\cite{N59}, but Bocci and Harbourne compute that its resurgence is $\rho(I)=\frac{s+1}{k}$~\cite[Corollary~1.3.1]{BH10} (this also coincides with the asymptotic resurgence).  The smallest example of this is the ideal of $36=\binom{9}{2}$ general points in $\mathbb{P}^2$, which has resurgence of $\rho(I)=\frac{8}{6}=\frac{4}{3}$.
\end{remark}

\begin{remark}\label{rem:IrrationalResurgence}
Remark~\ref{rem:RationalResurgence} also presents a natural place to look for \textit{irrational} resurgence.  Bocci and Harbourne explain in~\cite{BH10} that a (still open) conjecture of Nagata would imply that the ideal $I$ of $d=\binom{s+2}{2}$ generic points in $\mathbb{P}^2$ has resurgence $\rho(I)=\frac{s+1}{\sqrt{d}}$ for $n>9$.  (As Remark~\ref{rem:RationalResurgence} indicates, this is known if $d$ is square.)
\end{remark}

\begin{remark}\label{rem:NoetherianSymbolicReesAlgebra}
In general it is quite difficult to determine if the symbolic Rees algebra of an ideal is Noetherian, but this is known for some interesting classes of ideals.  In~\cite{Lyu88} Lyubeznik shows that $R_s(I)$ is Noetherian if $I$ is a squarefree monomial ideal.  This is extended to arbitrary monomial ideals by Herzog, Hibi, and Trung in~\cite{HHT07}.  The Noetherian property of the symbolic Rees algebra of a monomial curve has also been extensively studied.  An example of an ideal $I$ defining a monomial curve so that $R_s(I)$ is not Noetherian is provided in~\cite{MG92}.  Cutkosky gives several criteria for monomial curves in~\cite{Cut91}; perhaps the simplest of these is that if $P$ is the kernel of the homomorphism $k[x,y,z]\to k[t]$ with $x\to t^a,y\to t^b,z\to t^c$, then $R_s(P)$ is Noetherian if $(a+b+c)^2>abc$.
\end{remark}

\begin{remark}\label{rem:Q6}
	Following up from Proposition~\ref{prop:Resurgence} and the discussion at the beginning of Section~\ref{sec:main}, it seems natural to ask when Equation~\eqref{eq:finite} holds.  The following example, which comes from combinatorial optimization, shows that Equation~\eqref{eq:finite} can hold even when $\rho(I)=\widehat{\rho}(I)=1$.
	
	Consider the squarefree monomial ideal $I=\langle abd,ace,bcf,def\rangle$ in the polynomial ring $K[a,b,c,d,e,f]$.  One can show, for instance using vertex cover algebras from~\cite{HHT07}, that $I^{(r)}=I^r+(abcdef)I^{r-2}$ for $r\ge 2$.  From this equality it is straightforward to show that $I^{(r)}\neq I^r$ for $r\ge 2$ but $I^{(r+1)}\subset I^r$ for $r\ge 1$.  Thus $\rho(I)=1$, and is attained as a maximum (indeed we can take $M=N=2$ in Equation~\eqref{eq:finite}).  This example also provides a negative answer to~\cite[Question~5.2]{BH10b}.
\end{remark}

\begin{remark}\label{rem:NonPolynomialRings}
Our results hold in more generality than for polynomial rings over a field.  The arguments of~\cite[Section~4]{DFMS18} are stated for an ideal $I$ in a polynomial ring $R$, but all that is needed is module-finiteness of integral closures (giving an integer $k$ so that $\overline{I^{r+k}}\subset I^r$ for all $r\ge 1$) and the equivalence of the symbolic and adic topologies of the ideal $I$.  (For the existence of Rees valuations we only need $R$ to be Noetherian~\cite[Theorem~10.2.2]{IntegralClosure}.)  Theorem~\ref{thm:AsymptoticResurgenceByIntegralClosures}, and hence Theorem~\ref{thm:Main}, requires no additional hypotheses beyond these.  For most results of Section~\ref{sec:Expected} we need $R$ to be a regular ring containing a field as this is required to apply~\cite[Theorem~4.4]{J14}.
\end{remark}


\begin{remark}\label{rem:analyticspread}
The results in Section~\ref{sec:Expected} can be extended to non-radical ideals if the big height $h$ of the ideal $I$ is replaced by the largest analytic spread of $IR_P$ for any associated prime $P$ of $I$.  This is what Johnson denotes by $h$ in~\cite{J14}, and it coincides with the big height of $I$ if $I$ is radical.
\end{remark}

\begin{remark}[Additional bounds on asymptotic resurgence and resurgence]
Let $I$ be a radical ideal with big height $h$ and let $b(I)$ be the smallest of the integers $k$ so that $\overline{I^{r+k}}\subset I^r$ (the Brian\c{c}on-Skoda theorem implies that $b(I)\le n-1$, where $n$ is the number of variables of $R$).  If $I^{(s)}\not\subset I^r$ then $I^{(s)}\not\subset \overline{I^{r+b(I)}}$ and, applying part (1) of~\cite[Lemma~4.12]{DFMS18}, $\frac{s}{r+b(I)}<\wrho(I)$.  Conversely, if $I^{(s)}\subset I^r$ then $\wrho(I)\le \frac{s-1+h}{r}$ by Theorem~\ref{thm:AsymptoticResurgenceBound}.  For a fixed $r$, suppose $\lambda_r$ is the largest of the integers $s$ with $I^{(s)}\not\subset I^r$, so $I^{(\lambda_r)}\not\subset I^r$ and $I^{(\lambda_r+1)}\subset I^r$.  Then
\[
\frac{\lambda_r}{r+b(I)}<\wrho(I)\le \frac{\lambda_r+h}{r}.
\]
Thus $\wrho(I)$ can be approximated arbitrarily accurately provided that one can compute the integers $\lambda_r$.  (The difference between the bounds is on the order of $h/r$, so the bounds convergence to the asymptotic resurgence fairly slowly).  In particular, we see that $\wrho(I)=\lim_{r\to\infty}\frac{\lambda_r}{r}$.

Now for any positive integer $r$ put $K_r=\min\{s:\overline{I^s}\subset I^r\}$ and let $K(I)=\max\{\frac{K_r}{r}\}$.  By~\cite[Proposition~4.19]{DFMS18}, $\rho(I)\le\wrho(I)K(I)$.  With $b(I)$ defined as above, $K_r\le r+b(I)$.  Since $I$ is radical, it is integrally closed, so $K_1=1$.  Thus $K(I)\le 1+\frac{b(I)}{2}$.  Clearly if $I^{(s)}\not\subset I^r$ then $\frac{s}{r}\le \rho(I)$.  If $I^{(s)}\subset I^r$ then $\wrho(I)\le \frac{s-1+h}{r}$ and $\rho(I)\le \frac{s-1+h}{r}(1+\frac{b(I)}{2})\le \frac{s-1+h}{r}(1+\frac{n-1}{2})$, where $n$ is the number of variables of $R$.  Again letting $\lambda_r$ be as above, we have
\[
\frac{\lambda_r}{r}\le \rho(I)\le \frac{\lambda_r+h}{r}\left(1+\frac{b(I)}{2}\right).
\]
We have $\rho(I)=\sup\{\lambda_r/r\}$ by definition of $\rho(I)$, but we should not expect $\rho(I)=\inf\{(\lambda_r+h)/r (1+b(I)/2)\}=\wrho(I)(1+b(I)/2)$; for instance if $\rho(I)=\wrho(I)$ then this value will be off by a factor of $(1+b(I))/2$.  Even if $\wrho(I)<\rho(I)$ this can be too large.  Consider the first ideal $I$ of Example~\ref{ex:resurgencefromasymptoticresurgence}, where $\rho(I)=3/2,\wrho(I)=9/7,$ and $b(I)=1$.  Then $\wrho(I)(1+b(I)/2)=(9/7)(3/2)>\rho(I)$.
\end{remark}

\begin{remark}\label{rem:improvement}
Let $I\subset R$ be a radical ideal with big height $h$ and $J$ the intersection of all embedded primes of $\overline{I^i}$ for some $i\gg 0$.  Using Corollary~\ref{cor:expected} and the same techniques as in the proof of Theorem~\ref{thm:asymptoticresurgencebound2}, we can recover~\cite[Corollary~2.8]{GHM19} with the following slightly weaker hypotheses: if there is an increasing sequence of positive integers $\{r_i\}_{i=1}^\infty$, a sequence of positive integers $\{\gamma_i\}_{i=1}^\infty$, and a non-negative sequence of integers $\{f_i\}_{i=1}^\infty$ satisfying
\[
\lim\limits_{i\to\infty} \dfrac{\gamma_i}{r_i}=1,\;\;\; \lim\limits_{i\to \infty} \dfrac{f_i}{r_i}=\lim\limits_{i\to\infty} \dfrac{f_i}{\gamma_i}>0,\;\;\; \mbox{ and } I^{(\gamma_i h)}\subset J^{f_i} \overline{I^{r_i}},
\]
then $\rho(I)<h$.
\end{remark}

\begin{remark}
We discuss a circle of ideas related to~\cite{BGTT20}, where it is shown that a single containment implies Chudnovsky's conjecture for a general set of points in $\mathbb{P}^n$.  Chudnovsky's conjecture~\cite{C81} for ideals of points in $\mathbb{P}^{\ell}$ is equivalent to the inequality $\widehat{\alpha}(I)\ge \frac{\alpha(I)+\ell-1}{\ell}$, where $\alpha:R=K[x_0,\ldots,x_\ell]\to\Z$ is the $R$-valuation defined by $\alpha(f)=\deg(f)$ on homogeneous polynomials.  Using Theorem~\ref{thm:AsymptoticResurgenceByIntegralClosures}, we see that a Chudnovsky-type inequality of the form $\widehat{\nu}(I)\ge \frac{\nu(I)+C}{h}$ for some integer $C$ and every $R$-valuation $\nu$ implies that the asymptotic resurgence of $I$ is bounded away from $h$ (see also~\cite[Question~4.13]{HKZ20}).  Following the lead of~\cite{HH13}, we see that an inequality of type $\widehat{\nu}(I)\ge \frac{\nu(I)+C}{h}$ for some positive integer $C$ is implied by a containment of the type $I^{(rh)}\subset J^{rC} \overline{I^r}$ for infinitely many $r> 0$, where $J$ is the intersection of all primes associated to the integeral closure of some power of $I$ which are not minimal primes of $I$ (as in Remark~\ref{rem:improvement}).

Following the same reasoning as in the proof of Theorem~\ref{thm:asymptoticresurgencebound2}, we can show the following.  Suppose $I^{(s+1)}\subset J^{rC} \overline{I^r}$ for some integers $s,r,C\ge 1$.  Then
\[
\widehat{\nu}(I)\ge \frac{rh}{s+h}\frac{\nu(I)+C}{h}.
\]
In particular, if $I^{(hr-h+1)}\subset J^{rC}\overline{I^r}$ for a single positive integer $r$ then $\widehat{\nu}(I)\ge \frac{\nu(I)+C}{h}$ for every $R$-valuation $\nu$ supported on $I$.  See also~\cite[Proposition~5.3]{BGTT20}.

\end{remark}

\begin{remark}
In each of the examples considered in this paper (for instance Examples~\ref{ex:resurgencefromasymptoticresurgence} and~\ref{ex:3triangles}) it is true that $\overline{I^{r+1}}\subset I^r$ for every $r\ge 1$.  As pointed out in~\cite[Question~4.6]{HKZ20}, it would be interesting to have an example of a radical ideal where $\overline{I^{r+1}}\not\subset I^r$ for some $r\ge 1$.
\end{remark}

\section{Acknowledgements}
We thank Daniel Hern\'{a}ndez for commenting that $\rho(I)$ might be a maximum of finitely many values if $\wrho(I)<\rho(I)$, Alexandra Seceleanu and Elo\'{i}sa Grifo for suggesting a connection to expected resurgence, and Alexandra Seceleanu and Brian Harbourne for helpful comments on earlier drafts.  We also thank Vivek Mukundan, Adam Van Tuyl, Chris Francisco, Jay Schweig, Jeff Mermin, and Elena Guardo for their comments.

%


\end{document}